\documentclass[10pt,a4paper]{article}
\usepackage{fullpage}
\usepackage{amsfonts,amsmath,amssymb}
\usepackage{amsthm}
\usepackage{graphicx}\usepackage{sectsty}
\usepackage{authblk}

\theoremstyle{plain}
\newtheorem{theorem}{Theorem}[section]
\newtheorem{lemma}[theorem]{Lemma}

\theoremstyle{definition}
\newtheorem{definition}[theorem]{Definition}
\theoremstyle{remark}

\numberwithin{equation}{section}

\sectionfont{\large}
\newenvironment{data availability}[1][Data Availability
]{\begin{trivlist} \item[\hskip \labelsep {\bfseries
#1}]}{\end{trivlist}}
\makeatletter
\newcommand{\opnorm}{\@ifstar\@opnorms\@opnorm}
\newcommand{\@opnorms}[1]{%
  \left|\mkern-1.5mu\left|\mkern-1.5mu\left|
   #1
  \right|\mkern-1.5mu\right|\mkern-1.5mu\right|
}
\newcommand{\@opnorm}[2][]{%
  \mathopen{#1|\mkern-1.5mu#1|\mkern-1.5mu#1|}
  #2
  \mathclose{#1|\mkern-1.5mu#1|\mkern-1.5mu#1|}
}
\makeatother
\begin{document}
\title{Wave Support Theorem \\and Inverse Resonant Uniqueness on the Line}
\author{Lung-Hui Chen$^1$}\maketitle\footnotetext[1]{General Education Center, Ming Chi University of Technology, New Taipei City, 24301, Taiwan. Email:
mr.lunghuichen@gmail.com.}\maketitle
\begin{abstract}
In the paper, we experimentally study the inverse problem with the resonant scattering determinant. We analyze the structure of characteristics of perturbed linear waves. Assuming there is the common part of potential perturbation propagating along the same strips, we estimate the common part of the perturbed wave, and its Fourier transform.
 We deduce the partial inverse uniqueness from the Nevanlinna type of representation theorem.
\end{abstract}
\section{Introduction}
Let us consider the Schr\"{o}dinger equation
\begin{equation}\label{1.1}
-\frac{d^{2}}{dx^{2}}+V(x),\,x\in\mathbb{R},\,V(x)\in L_{comp}^{1}([a,b]),\,a<0<b,\,|a|\ll|b|,
\end{equation}
where we assume the potential is effectively support on $[a,b]$, that is, $[a,b]$ is the minimal convex hull that contains the support of $V^{j}$.
The scattering matrix is of the form
\begin{equation}\label{S}
S(k)=\left(\begin{array}{cc}\frac{ik}{\hat{X}(k)} & \frac{\hat{Y}(k)}{\hat{X}(k)} \vspace{7pt}\\\frac{\hat{Y}(-k)}{\hat{X}(k)} & \frac{ik}{\hat{X}(k)}\end{array}\right).
\end{equation}
The scattering matrix $S(k)$ is meromorphic in $\mathbb{C}$, and its poles in $\{\Im k>0\}$ are the square roots of $L^{2}$-eigenvalues of~(\ref{1.1}). In this paper, we understand $\hat{X}(k)$ through the one-dimensional wave equation
\begin{eqnarray} \label{AA}
\left\{%
\begin{array}{ll}
\big(D_{x}^{2}-D_{y}^{2}+V(x)\big)A_{\pm}(x,y)=0;\vspace{5pt}\\
A_{\pm}(x,y)=\delta(x-y),\,\pm x\gg0,
\end{array}%
\right.
\end{eqnarray}
where $D_{y}A_{-}(x,y)=X(y-x)+Y(y+x)$. 
In particular \cite[p.\,727]{Mellin}, 
\begin{eqnarray}\label{1122}
&&X(x)-\delta'(x)+\frac{\int V(t)dt}{2}\delta(x)\in L^{1}(\mathbb{R})\cap L^{\infty}(\mathbb{R});\\
&&Y(y)-\frac{V(y/2)}{4}\in L^{1}(\mathbb{R})\cap L^{\infty}(\mathbb{R}).\label{511}
\end{eqnarray}
Thus, $A_{\pm}(x,y)$ satisfies the wave equation with $x$ taking the place of time (this choice is dictated by the forcing condition imposed) in~(\ref{AA}). The uniqueness part follows from the energy estimates of the wave equation \cite{Dya,Mellin,Tang,Zworski2}. 

\par
In this paper, we consider the complex analysis of entire function $\hat{X}(k)$ and $\hat{Y}(k)$  which are represented in form of
\begin{eqnarray}\label{1.6}
&&\hat{X}(k)=\int_{-2(b-a)}^{0}X(x)e^{-ikx}dx;\\\label{1.7}
&&\hat{Y}(k)=\int_{2a}^{2b}Y(y)e^{-iky}dy,\,k\in\mathbb{C},
\end{eqnarray}
that the unitary identity holds in $\mathbb{C}$:
\begin{equation}\label{U}
\hat{X}(k)\hat{X}(-k) = k^{2} + \hat{Y}(k)\hat{Y}(-k) ,
\end{equation}
\par
More importantly, we consider experimentally 
\begin{equation}\label{SS}
\det S(k):=\frac{-\hat{X}(-k)}{\hat{X}(k)}.
\end{equation}
as scattering data in this paper inspired its simpler analytic structure. 
\par
We define for potential $V^{j}$, $j=1,2$,\,$0<r\ll1$,
\begin{eqnarray}\label{1.10}
&&\hat{X}^{j}_{1}(k):=\mathcal{F}\{X^{j}\chi_{[2a,0]}(y-b)\}(k);\\\label{1.11}
&&\hat{X}^{j}_{2}(k):=\mathcal{F}\{X^{j}\chi_{[2a-2r,2a]}(y-b)\}(k);\\\label{1.12}
&&\hat{X}^{j}_{3}(k):=\mathcal{F}\{X^{j}\chi_{[2a-2b,2a-2r]}(y-b)\}(k);\\
&&\hat{Y}^{j}_{1}(k):=\mathcal{F}\{X^{j}\chi_{[2a,0]}(y+b)\}(k);\\
&&\hat{Y}^{j}_{2}(k):=\mathcal{F}\{X^{j}\chi_{[0,2r]}(y+b)\}(k);\\
&&\hat{Y}^{j}_{3}(k):=\mathcal{F}\{X^{j}\chi_{[2r,2b]}(y+b)\}(k),
\end{eqnarray}
in which $\chi_{[2a,0]}(x)$ is the characteristic function defined on $[2a,0]$, and so on. The support of linear waves $D_{y}A^{j}_{-}(x,y)$ is illustrated in the shaded areas in Figure \ref{F}. Most importantly, $X^{j}\chi_{[2a-2r,2a]}(y-b)$ in not in the domain of influence of $V^{j}\chi[a,0]$ when time variable $x= b$. The function $D_{y}A^{j}_{-}(x,y)=X(y-x)+Y(y+x)$ satisfies the wave equation for $x\geq b$. For $x\leq a$, we have $D_{y}A^{j}_{-}(x,y)=-\delta'(x-y)$. If $b=0$, we firstly see the support of $X^{j}\chi_{[2a,0]}(y)$. For $b\gg0$, we find the red triangle in the center of diagram that shows the support of the wave solution that is not affected by the $V^{j}\chi[a,0]$, and simply depends on $V^{j}\chi[0,b]$. We refer more detailed analysis construction and characteristics analysis to Figure \ref{F}.
\begin{figure}\begin{center}\includegraphics[scale=0.23]{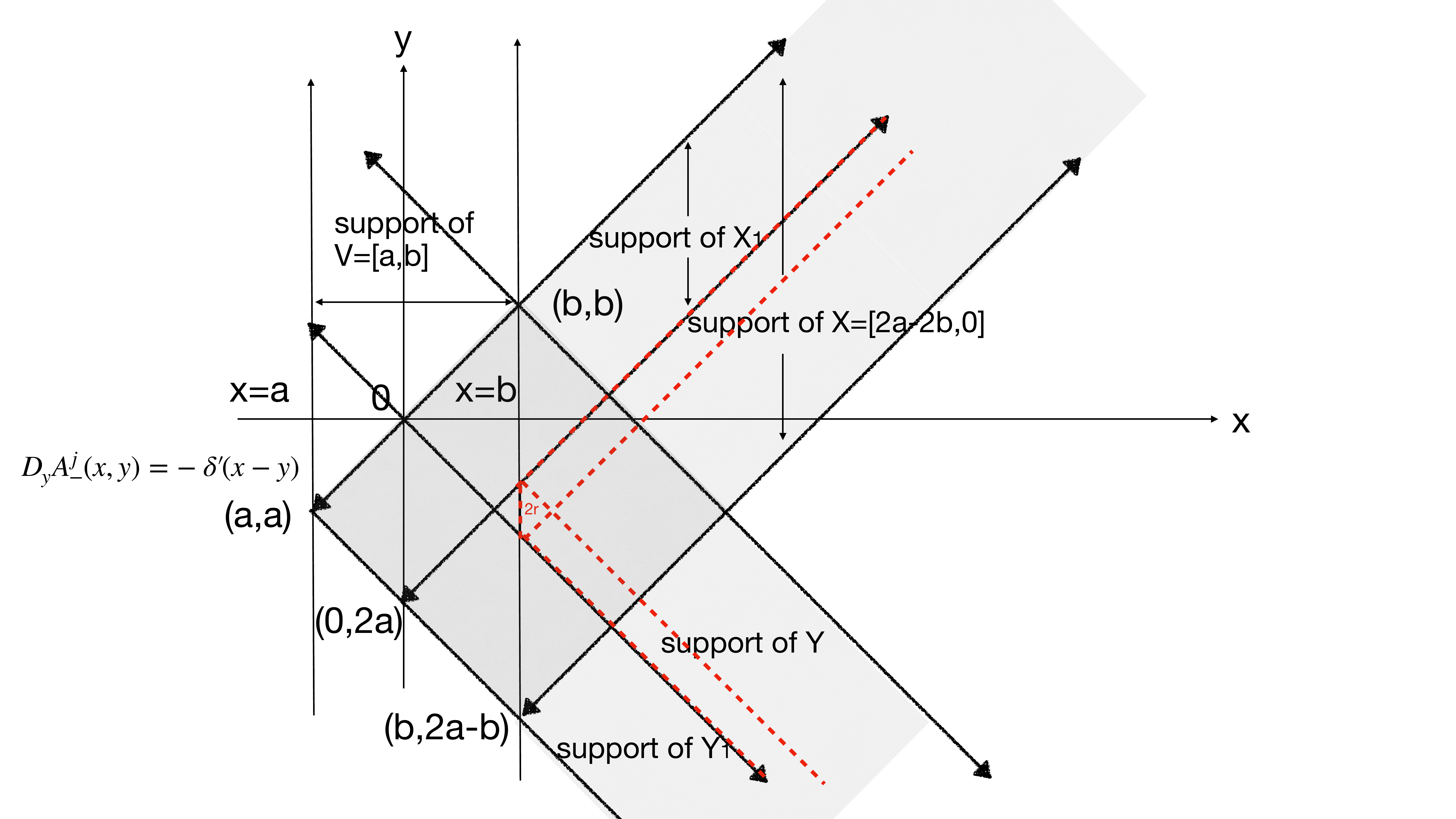}\caption{characteristics of linear waves}\label{F}\end{center}\end{figure}
\par
Let 
$$\det S^{j}(k):=\frac{-\hat{X}^{j}(-k)}{\hat{X}^{j}(k)},\,j=1,2,
$$
be the corresponding scattering determinant of of potential $V^{j}$.
\begin{theorem}\label{11}
If $V^{1}(x)\equiv V^{2}(x)$ on non-empty $[0,b]\subset[a,b]$, and $\det S^{1}(k)=\det S^{2}(k)$, then $V^{1}(x)\equiv V^{2}(x)$ on $[a,b]$.
\end{theorem}
In literature, when there are no bound states, the potential is determined by the reflection coefficients by Faddeev's theory \cite{DT79,Fa64}. 
In inverse resonance problem, we consider to determine the potential $V$ from the resonances of~(\ref{1.1}) which includes the square root of $L^{2}$-eigenvalues. The inverse resonance problem of Schr\"{o}dinger operator on the half line has been studied in \cite{Korotyaev1,Korotyaev3,Zworski2}. In the half line case, the unique recovery of the potential from the eigenvalues and resonances is justified in \cite{Korotyaev1}. However, in the full line case, the inverse resonance problems mainly remained open for a long time. It is known that the potential cannot be solely determined by the eigenvalues and resonances. Specifically, Zworski \cite{Zworski2} proved the uniqueness theorem for the symmetric potentials along with certain isopolar results. Furthermore, Korotyaev \cite{Korotyaev1,Korotyaev3} applied the value distribution theory in complex analysis to prove that all eigenvalues and resonances, and a signed sequence can uniquely determine the potential $V$.

\section{Lemmata}
\begin{lemma}\label{21}
If $F(x)$ is of bounded variation on $(-\infty,\infty)$, then $F(x)$ is constant except on some finite interval if and only if $f(z)=\int_{-\infty}^{\infty}e^{izt}dF(t)$ is an entire functional of exponential type; and if $(a,b)$ is the smallest interval outside which $F(x)$ is a constant, then $a=h_{f}(-\frac{\pi}{2})$ and $b=h_{f}(\frac{\pi}{2})$.
\end{lemma}
\begin{proof}
We refer the proof to Boas \cite[p.\,109]{Boas}.
\end{proof}
Here, we note that the indicator diagram of $\hat{X}^{j}(k)$ is the line set $i[-2(b-a),0]$ on the imaginary axis, and its length is $|a|+|b|$.
\begin{lemma}
The length of indicator diagram of $\hat{X}^{j}(k)$ is $2(b-a)$.
\end{lemma}
\begin{proof}
We use Lemma \ref{21} and~(\ref{1.6}) to conclude the result.
\end{proof}
\begin{lemma}
$\hat{X}^{j}_{1}(k)$ has indicator function $|a||\sin\theta|$;\,$\hat{X}^{j}_{2}(k)$ has indicator function $|r||\sin\theta|$;\,$\hat{X}^{j}_{3}(k)$ has indicator function $|b-r||\sin\theta|$.
\end{lemma}
\begin{proof}
It is straightforward by Lemma \ref{21},~(\ref{1.10}),~(\ref{1.11}), and~(\ref{1.12}).
\end{proof}
\begin{lemma}\label{24}
$\hat{X}^{j}_{j'}(k)$ has only finite zeros in $\mathbb{C}^{+}$, and infinitely many in $\mathbb{C}^{-}$.
\end{lemma}
\begin{proof}
This is well-known in the literature, say, \cite{Dya}.
\end{proof}
\begin{lemma}\label{id}
If $V^{1}\equiv V^{2}$ on $[0,b]$, then $\hat{X}^{1}_{2}(k)\equiv\hat{X}^{1}_{2}(k)$, and $\hat{Y}^{1}_{2}(k)\equiv\hat{Y}^{1}_{2}(k)$.
\end{lemma}
\begin{proof}
Let us discuss the Figure \ref{F}. Let us first take $b=0$, so the wave starting inside the triangle: $(0,0)$, $(a,a)$, $(0,2a)$ propagate in parallel strips between the shaded area between the lines pointing at $(a,a)$ and $(0,2a)$ and the lines pointing at $(a,a)$ and $(0,0)$. The wave function $X^{j}_{1}$ and $Y^{j}_{1}$, $j=1,2$, are defined on the those two strips correspondingly. Now we consider $x\in[0,b]$, if $V^{1}_{2}\equiv V^{2}_{2}$ on $[0,b]$,  then the wave function $Y^{1}_{2}\equiv Y^{2}_{2}$ are defined on the strips between the lines contain $(0,0)$ and $(b,b)$. Similarly, the wave function $X^{1}_{2}\equiv X^{2}_{2}$ are defined on the strips between the lines contain $(0,2a)$ and $(b,2a-b)$. Now we deduce from~(\ref{1.10}) that $\hat{X}^{1}_{2}(k)\equiv\hat{X}^{1}_{2}(k)$. Similar arguments works for $\hat{Y}^{1}_{2}(k)\equiv\hat{Y}^{1}_{2}(k)$.

\end{proof}

\section{Proof of Theorem \ref{11}}

We start with the assumption of Theorem \ref{1.1}
$$\det S^{1}(k)\equiv\det S^{2}(k),$$
that is 
\begin{equation}\label{3.1}
\frac{\hat{X}^{1}(-k)}{\hat{X}^{1}(k)}\equiv\frac{\hat{X}^{2}(-k)}{\hat{X}^{2}(k)}.
\end{equation}
Using Lemma \ref{24} and comparing the poles on both sides of~(\ref{3.1}), $\hat{X}^{1}(k)$ and $\hat{X}^{2}(k)$ have common zeros in $\mathbb{C}^{-}$, except finite ones in $\mathbb{C}^{+}$.
If $\sigma$ is a common zero of $\hat{X}^{1}(k)$ and $\hat{X}^{2}(k)$, then 
\begin{equation}
\hat{X}^{1}_{1}(\sigma)
+\hat{X}^{1}_{3}(\sigma)=\hat{X}^{2}_{1}(\sigma)+\hat{X}^{2}_{3}(\sigma)=0,
\end{equation}
by Lemma \ref{id}. Now we want to show that
\begin{equation}\label{3.2}
\hat{X}^{1}(k):=\hat{X}^{1}_{1}(k)+\hat{X}^{1}_{2}(k)
+\hat{X}^{1}_{3}(k)\equiv\hat{X}^{2}_{1}(k)+\hat{X}^{2}_{2}(k)+\hat{X}^{2}_{3}(k)=:\hat{X}^{2}(k).
\end{equation}
That is, 
\begin{equation}\label{3.22}
\hat{X}^{1}_{1}(k)
+\hat{X}^{1}_{3}(k)\equiv\hat{X}^{2}_{1}(k)+\hat{X}^{2}_{3}(k).
\end{equation}
By Lemma \ref{id},~(\ref{3.2}) and~(\ref{3.22}) should have same density of common zeros.
\par
Let us count the zero set of
\begin{equation}
F(k):=\hat{X}^{j}(k),\,j=1,2,
\end{equation}
and the zero set of
\begin{equation}\label{3.3}
G(k):=\hat{X}^{1}_{1}(k)
+\hat{X}^{1}_{3}(k)-\hat{X}^{2}_{1}(k)-\hat{X}^{2}_{3}(k).
\end{equation}
Using ~(\ref{119}) and~(\ref{120}) in Appendix, we deduce that
\begin{eqnarray}
&&h_{F}(\theta)=(b-a)|\sin\theta|;\\
&&h_{G}(\theta)=\max\{|a|,|b-r|\}|\sin\theta|=|b-r||\sin\theta|.
\end{eqnarray}
Using Theorem \ref{C}, $F(z)$  has zero density $\frac{(b-a)}{\pi}$, and $G(k)$  has identical zero density $\frac{b-r}{\pi}$. That contradicts to Lemma \ref{id}. Then, we deduce that 
$$
\hat{X}^{1}_{1}(k)
+\hat{X}^{1}_{3}(k)\equiv\hat{X}^{2}_{1}(k)+\hat{X}^{2}_{3}(k)
.$$
Due to Lemma \ref{id}, we have
$$\hat{X}^{1}(k)\equiv\hat{X}^{2}(k).$$
\par
Using~(\ref{U}), we have
\begin{equation}\label{312}
\hat{X}^{j}(k)\hat{X}^{j}(-k)=k^{2}+\hat{Y}^{j}(k)\hat{Y}^{j}(-k).
\end{equation}
Thus, we obtain
\begin{equation}\label{313}
\hat{Y}^{1}(k)\hat{Y}^{1}(-k)\equiv\hat{Y}^{2}(k)\hat{Y}^{2}(-k),\,k\in\mathbb{C}.
\end{equation}
Equivalently,
$$|\hat{Y}^{1}(k)|^{2}=|\hat{Y}^{2}(k)|^{2},\,k\in\mathbb{R},$$
and we apply Theorem \ref{NL} to deduce
\begin{equation}\label{3100}
\hat{Y}^{1}(k)\prod_{n=1}^{\infty}\frac{1-\frac{k}{\overline{a}^{1}_{n}}}{1-\frac{k}{a^{1}_{n}}}=e^{i\gamma}\hat{Y}^{2}(k)\prod_{n=1}^{\infty}\frac{1-\frac{k}{\overline{a}^{2}_{n}}}{1-\frac{z}{a^{2}_{n}}},
\end{equation}
where $\{a_{n}^{j}\}$ are the zeros of $\hat{Y}^{j}(k)$ in $\mathbb{C}^{+}$. The Blaschke product
$$\prod_{n=1}^{\infty}\frac{1-\frac{k}{\overline{a}^{j}_{n}}}{1-\frac{k}{a^{j}_{n}}}$$ is a function of zero type and of zero density of zeros. We refer the detail to \cite{Boas}. Again, the zero density of~(\ref{3100}) is $\frac{b-a}{\pi}$. Thus, $\hat{Y}^{1}(k)$ and $\hat{Y}^{2}(k)$ have common zero of density $\frac{b-a}{\pi}$. That is,
\begin{equation}
\hat{Y}^{1}(\sigma)=\hat{Y}^{2}(\sigma),\,\forall\sigma\in\Sigma,
\end{equation}
in which $\Sigma$ is a set of density $\frac{b-a}{\pi}$. That is,
\begin{equation}
\hat{Y}^{1}_{1}(\sigma)+\hat{Y}^{1}_{2}(\sigma)+\hat{Y}^{1}_{3}(\sigma)=\hat{Y}^{2}_{1}(\sigma)+\hat{Y}^{2}_{2}(\sigma)+\hat{Y}^{2}_{3}(\sigma),\,\forall\sigma\in\Sigma.
\end{equation}
That is, by Lemma \ref{id},
\begin{equation}
\hat{Y}^{1}_{1}(\sigma)+\hat{Y}^{1}_{3}(\sigma)=\hat{Y}^{2}_{1}(\sigma)+\hat{Y}^{2}_{3}(\sigma),\,\forall\sigma\in\Sigma,
\end{equation}
which however could have a zero set of density $\frac{b-a-r}{\pi}$.
Hence, $\hat{Y}^{1}_{1}(k)+\hat{Y}^{1}_{3}(k)\equiv\hat{Y}^{2}_{1}(k)+\hat{Y}^{2}_{3}(k)$, and then
$$\hat{Y}^{1}(k)\equiv\hat{Y}^{2}(k).$$
Therefore, we deduce that $V^{1}$ and $V^{2}$ have the same scattering matrix;
\begin{equation}
S^{1}(k)\equiv S^{2}(k).
\end{equation}
Using Zworski \cite[Proposition\,8]{Zworski}, which says that the potential function with compact support is determined by the scattering matrix, we deduce that 
\begin{eqnarray}
V^{1}(x)\equiv V^{2}(x).
\end{eqnarray}
This proves the theorem.

\section{Appendix}
We review some results from complex analysis  \cite{Boas,Levin2}.
\begin{definition}
Let $f(z)$ be an entire function. Let
\begin{equation}\nonumber
M_f(r):=\max_{|z|=r}|f(z)|.
\end{equation}
An entire function of $f(z)$ is said
to be a function of finite order if there exists a positive
constant $k$ such that the inequality
\begin{equation}\nonumber
M_f(r)<e^{r^k}
\end{equation}
is valid for all sufficiently large values of $r$. The greatest
lower bound of such numbers $k$ is called the order of the entire
function $f(z)$. By the type $\sigma$ of an entire function $f(z)$
of order $\rho$, we mean the greatest lower bound of positive
number $A$ for which asymptotically we have
\begin{equation}\nonumber
M_f(r)<e^{Ar^\rho}.
\end{equation}
That is,
\begin{equation}\nonumber
\sigma_{f}:=\limsup_{r\rightarrow\infty}\frac{\ln M_f(r)}{r^\rho}.
\end{equation}  If $0<\sigma_{f}<\infty$, then we say
$f(z)$ is of normal type or mean type. For $\sigma_{f}=0$, we say $f(z)$ is of minimal type.
\end{definition}
We refer the details to \cite{Levin2}.
\begin{definition}\label{33}
Let $f(z)$ be an integral function of finite order $\rho$ in the
angle $[\theta_1,\theta_2]$. We call the following quantity as the
indicator function of function $f(z)$.
\begin{equation}\nonumber
h_f(\theta):=\lim_{r\rightarrow\infty}\frac{\ln|f(re^{i\theta})|}{r^{\rho}},
\,\theta_1\leq\theta\leq\theta_2.
\end{equation}
\end{definition}
The type of a function is connected to the maximal value of indicator function.
\begin{definition}\label{d}
The following quantity is called the width of the indicator
diagram of entire function $f$:
\begin{equation}\label{22222}
d=h_f(\frac{\pi}{2})+h_f(-\frac{\pi}{2}).
\end{equation}
\end{definition}

\begin{definition}\label{255}
Let $f(z)$ be an integral function of order $1$, and let
$n(f,\alpha,\beta,r)$ denote the number of the zeros of $f(z)$
inside the angle $[\alpha,\beta]$ and $|z|\leq r$. We define the
density function as
\begin{equation}\nonumber
\Delta_f(\alpha,\beta):=\lim_{r\rightarrow\infty}\frac{n(f,\alpha,\beta,r)}{r},
\end{equation}
and
\begin{equation}\nonumber
\Delta_f(\beta):=\Delta_f(\alpha_0,\beta),
\end{equation}
with some fixed $\alpha_0\notin E$ such that $E$ is at most a
countable set \cite{Boas,Levin,Levin2}. In particular, we denote the density function of $f$ on the open right/left half complex plane as $\Delta^{+}_{f}$/$\Delta^{-}_{f}$ respectively. Similarly, we can define the set density of a zero set $S$. Let $n(
S,r)$ be the number of the discrete elements of $S$ in $\{|z|<r\}$. We define
\begin{equation}\nonumber
\Delta_S:=\lim_{r\rightarrow\infty}\frac{n(S,r)}{r},
\end{equation}

\end{definition}
\begin{theorem}[Cartwright]\label{C}
Let $f$ be an entire function of exponential type with zero set $\{a_{k}\}$. We assume $f$ satisfies one of the
following conditions:
\begin{equation}\nonumber
\mbox{ the integral
}\int_{-\infty}^\infty\frac{\ln^+|f(x)|}{1+x^2}dx\mbox{ exists}.
\end{equation}
\begin{equation}\nonumber
|f(x)|\mbox{ is bounded on the real axis}.
\end{equation}
Then
\begin{enumerate}
\item all of the zeros of the function $f(z)$, except possibly
those of a set of zero density, lie inside arbitrarily small
angles $|\arg z|<\epsilon$ and $|\arg z-\pi|<\epsilon$, where the
density
\begin{equation}
\Delta_f(-\epsilon,\epsilon)=\Delta_f(\pi-\epsilon,\pi+\epsilon)=\lim_{r\rightarrow\infty}
\frac{n(f,-\epsilon,\epsilon,r)}{r}
=\lim_{r\rightarrow\infty}\frac{n(f,\pi-\epsilon,\pi+\epsilon,r)}{r},
\end{equation}
is equal to $\frac{d}{2\pi}$, where $d$ is the width of the
indicator diagram in~(\ref{22222}). Furthermore, the limit
$\delta=\lim_{r\rightarrow\infty}\delta(r)$ exists, where
$$
\delta(r):=\sum_{\{|a_k|<r\}}\frac{1}{a_k};
$$
\item moreover,
\begin{equation}\nonumber
\Delta_f(\epsilon,\pi-\epsilon)=\Delta_f(\pi+\epsilon,-\epsilon)=0;
\end{equation}
\item the function $f(z)$ can be represented in the form
\begin{equation}\nonumber
f(z)=cz^me^{i\kappa
z}\lim_{r\rightarrow\infty}\prod_{\{|a_k|<r\}}(1-\frac{z}{a_k}),
\end{equation}
where $c,m,\kappa$ are constants and $\kappa$ is real;
\item the indicator
function of $f$ is of the form
\begin{equation}
h_f(\theta)=\sigma|\sin\theta|.
\end{equation}
\end{enumerate}
\end{theorem}
We refer the Cartwright's theory to Levin \cite[p,251]{Levin}.
\begin{lemma}\label{36}
Let $f$, $g$ be two entire functions. Then the following two
inequalities hold.
\begin{eqnarray}
&&h_{fg}(\theta)\leq h_{f}(\theta)+h_g(\theta),\mbox{ if one limit exists};\label{119}\\\label{120}
&&h_{f+g}(\theta)\leq\max_\theta\{h_f(\theta),h_g(\theta)\},
\end{eqnarray}
where the equality in~(\ref{119}) holds if one of the functions is of completely regular growth, and secondly the equality~(\ref{120}) holds if the indicator of the two summands are not equal at some $\theta_0$.
\end{lemma}
\begin{proof}
 We can find
the details in \cite{Levin}.
\end{proof}

\begin{lemma}\label{38}
The Fourier transform $\hat{X}(z)$ as in~(\ref{1.6}) is of Cartwright class, and the function can be represented in the form
$$\hat{X}(z)=cz^{m}e^{i\delta z}\lim_{R\rightarrow\infty}\prod_{|\sigma_{n}|<R}(1-\frac{z}{\sigma_{n}}),\,z=x+iy,$$
where $\delta\in\mathbb{R}$, and the following integral converges:
\begin{equation}\label{2100}
\int_{-\infty}^{\infty}\frac{\ln^{+}|\hat{X}(x)|}{1+x^{2}}dx<\infty.
\end{equation}
Similar results hold for $\hat{Y}(k)$.
\end{lemma}
\begin{proof}
We refer the definition of Cartwright class to \cite{Levin,Levin2}.
\end{proof}
\begin{theorem}[Nevanlinna-Levin]\label{NL}
If the function $F(z)$ is holomorphic and of exponential type in the half-plane $\Im z\geq0$, and if~(\ref{2100}) holds, then
\begin{enumerate}
\item \begin{equation}\nonumber
F(z)\prod_{k=1}^{\infty}\frac{1-\frac{z}{\overline{a}_{k}}}{1-\frac{z}{a_{k}}}=e^{i\gamma}e^{u(z)+iv(z)},
\end{equation}
where $$u(z)=\frac{y}{\pi}\int_{-\infty}^{\infty}\frac{\ln|F(t)|}{(t-x)^{2}+y^{2}}dt+\sigma^{+}_{F} y,$$  $\sigma^{+}_{F}=h_{F}(\frac{\pi}{2})$, $v(z)$ is the harmonic conjugate of $u(z)$, and $\{a_{k}\}$ are the zeros of the function $F(z)$ in the half-plane $\Im z>0$;
\item
 $$
\ln|F(z)|=\frac{y}{\pi}\int_{-\infty}^{\infty}\frac{\ln|F(t)|}{(t-x)^{2}+y^{2}}dt+\sigma^{+}_{F} y+\ln|\chi(z)|,\,z=x+iy,
$$
where
$$\chi(z)=\prod_{k=1}^{\infty}\frac{1-\frac{z}{a_{k}}}{1-\frac{z}{\overline{a}_{k}}}.$$
\end{enumerate}
\end{theorem}
\begin{proof}
We refer the proof to \cite[p.\,240]{Levin}. 
\end{proof}

\end{document}